\documentclass[12pt, reqno]{amsart}
\usepackage{amsmath, amsthm, amscd, amsfonts, amssymb, graphicx, color}
\usepackage[bookmarksnumbered, colorlinks, plainpages]{hyperref}

\input{mathrsfs.sty}

\newtheorem{theorem}{Theorem}[section]
\newtheorem{lemma}[theorem]{Lemma}
\newtheorem{proposition}[theorem]{Proposition}
\newtheorem{corollary}[theorem]{Corollary}
\theoremstyle{definition}

\newtheorem{example}[theorem]{Example}

\theoremstyle{remark}
\newtheorem{remark}[theorem]{Remark}
\numberwithin{equation}{section}
\begin{document}

\title [Refinements of a reversed AM--GM operator inequality]{Refinements of a reversed AM--GM operator inequality}

\author[M. Bakherad ]{Mojtaba Bakherad }

\address{ Department of Mathematics, Faculty of Mathematics, University of Sistan and Baluchestan, P.O. Box 98135-674, Zahedan, Iran.}
\email{mojtaba.bakherad@yahoo.com; bakherad@member.ams.org}


%~~~~~~~~~~~~~~~~~~~~~~~~~~~~~~~~~~~~~~~~~~~~~~~~~~~~~~~~~~~~~~~~~~~~~~~~~~~~~~~~~~~~~~~~~~~~~~~~~~~~~~~~~~~~~~~~~~~~~~~~~~~~~~~~~~~
\begin{abstract}
We prove some refinements of a reverse AM--GM operator inequality
due to M. Lin [Studia Math. 2013;215:187--194]. In particular, we
show the operator inequality
 \begin{eqnarray*}
\Phi^p\left(A\nabla_\nu B+2rMm(A^{-1}\nabla B^{-1}-A^{-1}\sharp B^{-1})\right)\leq\alpha^p\Phi^p\left(A\sharp_\nu B\right),
\end{eqnarray*}
where $A,B$ are positive operators on a Hilbert space such that $0<m \leq A, B \leq M$ for some positive numbers $m, M$, $\Phi$ is a positive unital linear map,
 $\nu\in[0,1]$, $r=\min\{\nu,1-\nu\}$,
 $p>0$ and $\alpha=\max\left\{\frac{(M+m)^2}{4Mm},\frac{(M+m)^2}{4^\frac{2}{p}Mm}\right\}$.
\end{abstract} \maketitle

%~~~~~~~~~~~~~~~~~~~~~~~~~~~~~~~~~~~~~~~~~~~~~~~~~~~~~~~~~~~~~~~~~~~~~~~~~~~~~~~~~~~~~~~~~~~~~~~~~~~~~~~~~~~~~~~~~~~~~~
\section{Introduction and preliminaries}

Let ${\mathbb B}({\mathscr H})$ denote the $C^*$-algebra of all
bounded linear operators on a complex Hilbert space ${\mathscr
H}$, with the identity $I$. In the case when ${\rm dim}{\mathscr
H}=n$, we identify ${\mathbb B}({\mathscr H})$ with the matrix
algebra $\mathbb{M}_n$ of all $n\times n$ matrices with entries in
the complex field. An operator $A\in{\mathbb B}({\mathscr H})$ is
called positive if $\langle Ax,x\rangle\geq0$ for all
$x\in{\mathscr H }$, and we then write $A\geq0$. We write $A>0$ if
$A$ is a positive invertible operator. For self-adjoint operators
$A, B\in{\mathbb B}({\mathscr H})$ we say that $A\leq B$ if
$B-A\geq0$. The Gelfand map $f(t)\mapsto f(A)$ is an
isometrical $*$-isomorphism between the $C^*$-algebra
$C({\rm sp}(A))$ of continuous functions on the spectrum ${\rm sp}(A)$
of a self-adjoint operator $A$ and the $C^*$-algebra generated by $A$ and $I_{\mathscr H}$. If $f, g\in C({\rm sp}(A))$, then
$f(t)\geq g(t)\,\,(t\in{\rm sp}(A))$ implies that $f(A)\geq g(A)$.

Let $A,B\in{\mathbb B}({\mathscr H})$ be two positive invertible
operators and   $\nu\in[0,1]$.  The operator weighted arithmetic,
geometric and harmonic means are defined by $A\nabla_\nu
B=(1-\nu)A+\nu B, A\sharp_\nu
B=A^{\frac{1}{2}}\left(A^{\frac{-1}{2}}BA^{\frac{-1}{2}}\right)^{\nu}A^{\frac{1}{2}}$
and $A!_\nu B=\left((1-\nu)A^{-1}+\nu B^{-1}\right)^{-1}$,
respectively.  In particular, for $\nu=\frac{1}{2}$ we get the
usual operator arithmetic mean $\nabla$, the geometric mean
$\sharp$ and the harmonic mean $!$. The AM--GM inequality reads
\begin{align*}
\frac{A+B}{2}\geq A\sharp B,
\end{align*}
for all positive operators $A, B$. It
is shown in \cite{lin}  the following reverse of AM--GM inequality involving positive linear maps
\begin{align}\label{mos1}
\Phi\left(\frac{A+B}{2}\right)\leq \frac{(M+m)^2}{4Mm}\Phi(A\sharp B),
\end{align}
where $0<m \leq A,B\leq M$ and  $\Phi$ is a positive
unital linear map.

For two positive operators $A, B\in {\mathbb B}({\mathscr H})$,
the L\"owner--Heinz inequality states that, if $A\leq B$, then
\begin{align}\label{12345a}
A^p\leq B^p,\qquad(0\leq p\leq 1).
\end{align}
In general \eqref{12345a} is not true for $p>1$.  Lin
\cite[Theorem 2.1]{lin}  showed however a squaring of
\eqref{mos1}, namely that the inequality
\begin{align}\label{12345678}
\Phi^2\left(\frac{A+B}{2}\right)\leq \left(\frac{(M+m)^2}{4Mm}\right)^2\Phi^2(A\sharp B)
\end{align}
as well as
\begin{align}\label{123456781}
\Phi^2\left(\frac{A+B}{2}\right)\leq
\left(\frac{(M+m)^2}{4Mm}\right)^2(\Phi(A)\sharp \Phi(B))^2
\end{align}
hold. Using inequality \eqref{12345a} we therefore get
\begin{align}\label{lala1}
\Phi^p\left(\frac{A+B}{2}\right)\leq \left(\frac{(M+m)^2}{4Mm}\right)^p\Phi^p(A\sharp B)\qquad(0<p\leq 2)
\end{align}
and
\begin{align}\label{lala2}
\Phi^p\left(\frac{A+B}{2}\right)\leq \left(\frac{(M+m)^2}{4Mm}\right)^p(\Phi(A)\sharp \Phi(B))^p\qquad(0<p\leq 2),
\end{align}
where $0<m \leq A,B\leq M$ and $\Phi$ is a positive unital linear
map.

In \cite{f-h} the authors  extended  \eqref{12345678} and
\eqref{123456781} to $p>2$. They proved that the inequalities
\begin{align}\label{123456789a}
\Phi^p\left(\frac{A+B}{2}\right)\leq \left(\frac{(M+m)^2}{4^\frac{2}{p}Mm}\right)^p\Phi^p(A\sharp B)\qquad(p>2)
\end{align}
and
\begin{align}\label{123456789ab}
\Phi^p\left(\frac{A+B}{2}\right)\leq \left(\frac{(M+m)^2}{4^\frac{2}{p}Mm}\right)^p(\Phi(A)\sharp \Phi(B))^p\qquad(p>2),
\end{align}
where $0<m \leq A,B\leq M$. In \cite{hbt} and \cite{lin12} the authors showed that
\begin{eqnarray}\label{eq001}
\Phi^p\left(A\sigma
B\right)\leq\alpha^p\Phi^p\left(A\tau
B\right),
\end{eqnarray}
and
\begin{eqnarray}\label{eq002}
\Phi^p\left(A\sigma B\right)\leq\alpha^p\left(\Phi(A)\tau \Phi(B)\right)^p,
\end{eqnarray}
 where $0<m \leq A, B \leq M$,  $\Phi$ be a positive unital linear map, $\sigma$, $\tau$ be two arbitrary
means between harmonic and arithmetic means,  $\alpha=\max\left\{\frac{(M+m)^2}{4Mm},\frac{(M+m)^2}{4^\frac{2}{p}Mm}\right\}$ and $p>0$.
 Choi's inequality (see e.g. \cite[p. 41]{bk2}) reads
\begin{align}\label{coi}
\Phi(A)^{-1}\leq \Phi(A^{-1}),
\end{align}
for any positive unital linear map $\Phi$ and operator $A>0$.
Choi's inequality cannot be squared  \cite{lin}, but a reverse of
Choi's inequality (known as the operator Kantorovich inequality)
can be squared, see e.g. \cite{lin1}.

In this paper, we present some refinements of inequalities \eqref{lala1} and \eqref{lala2} under some mild conditions for $0<p\leq1$ and
 some refinements of inequalities \eqref{123456789a} and \eqref{123456789ab} for the operator norm and $p>2$. We also show a refinement of the operator P\'olya--Szeg\"o inequality.
%===================================================================================================================================
\section{Main results}
We need the following lemmas to prove our results.
\begin{lemma}\cite{bk}\label{main}
Let $A,B>0$. Then
\begin{align*}
\|AB\|\leq\frac{1}{4}\|A+B\|^2.
\end{align*}
\end{lemma}

\begin{lemma}\cite{ando}\label{main2}
Let $A,B\geq0$ and $p>1$. Then
\begin{align*}
\|A^p+B^p\|\leq\|(A+B)^p\|.
\end{align*}
\end{lemma}

\begin{lemma}\label{main3}
Let $A,B>0$ and $\alpha>0$. Then $A\leq\alpha B$ if and only if
$\|A^\frac{1}{2}B^\frac{-1}{2}\|\leq \alpha^\frac{1}{2}.$
\end{lemma}
\begin{proof}
Obviously, $A\leq\alpha B$ if and only if
$B^{\frac{-1}{2}}AB^{\frac{-1}{2}}\leq\alpha$.  By definition, this holds if
and only if $\|A^{\frac{1}{2}}B^{\frac{-1}{2}}\|^2 \leq\alpha$ and
the proof is complete.
\end{proof}
\begin{lemma}\cite{hbt}\label{male2}
Let $0<m \leq A, B \leq M$,  $\Phi$ be a positive unital linear map and $\sigma$, $\tau$ be two arbitrary
means between harmonic and arithmetic means. Then
\begin{eqnarray*}
\Phi(A\sigma B)+Mm\Phi^{-1}(A\tau B)\leq M+m.
\end{eqnarray*}
\end{lemma}

In the next proposition we extend the inequalities \eqref{eq001}
and \eqref{eq002} to $p>2$ and the inequalities \eqref{123456789a}
and \eqref{123456789ab} to arbitrary means between  harmonic and arithmetic means.

\begin{proposition}
Let $0<m \leq A, B \leq M$, $\Phi$ be a positive unital linear map, $\sigma$, $\tau$ be two arbitrary
means between  harmonic and arithmetic means and $p>0$. Then
\begin{eqnarray*}
\Phi^p\left(A\sigma B\right)\Phi^{-p}\left(A\tau
B\right)+\Phi^{-p}\left(A\tau B\right) \Phi^p\left(A\sigma
B\right)\leq2\alpha^p
\end{eqnarray*}
where $\alpha=\max\left\{\frac{(M+m)^2}{4Mm},\frac{(M+m)^2}{4^\frac{1}{p}Mm}\right\}$.
\end{proposition}
\begin{proof}
By \cite[Lemma 3.5.12]{horn} we have that $\|X\|\leq t$ if and
only if
 $\left(\begin{array}{cc}
 tI&X\\
 X^*&tI
 \end{array}\right)\geq0$, for any  $X\in{\mathbb B}({\mathscr
H})$. If $0<p\leq1$, then $\alpha=\frac{(M+m)^2}{4Mm}$. Applying inequality \eqref{eq001} and Lemma \ref{main3} we
get $$\|\Phi^p\left(A\sigma B\right)\Phi^{-p}\left(A\tau
B\right)\|\leq\alpha^{p}.$$
Hence
  $$\left(\begin{array}{cc}
 \alpha^{p}I&\Phi^p\left(A\sigma B\right)\Phi^{-p}\left(A\tau B\right)\\
 \Phi^{-p}\left(A\tau B\right)\Phi^p\left(A\sigma B\right)&\alpha^{p}I
 \end{array}\right)\geq0$$
  and
 $$\left(\begin{array}{cc}
 \alpha^{p}I&\Phi^{-p}\left(A\tau B\right)\Phi^p\left(A\sigma B\right)\\
 \Phi^p\left(A\sigma B\right)\Phi^{-p}\left(A\tau B\right)&\alpha^{p}I
 \end{array}\right)\geq0.$$
 Hence
 {\footnotesize\begin{align*}\left(\begin{array}{cc}
 {2}\alpha^{p}I&\Phi^{-p}\left(A\tau B\right)
 \Phi^p\left(A\sigma B\right)+\Phi^p\left(A\sigma B\right)\Phi^{-p}\left(A\tau B\right)\\
 \Phi^p\left(A\sigma B\right)\Phi^{-p}\left(A\tau B\right)+\Phi^{-p}\left(A\tau B\right)
 \Phi^p\left(A\sigma B\right)&{2}\alpha^{p}I
 \end{array}\right)
 \end{align*}}
 is positive and the desired inequality for $0<p\leq1$. Using inequality \eqref{eq001} with the same argument, we get the desired inequality for $p>1$.
\end{proof}
Now, we are ready to present our main result. We need the
following lemma, proved in \cite{kit}; (see also \cite{m-m21}).
\begin{lemma}\cite{kit}
Let $ a,b>0$ and $\nu\in[0,1]$. Then
\begin{align}\label{rif-y}
a^{1-\nu} b^\nu+r(\sqrt{a}-\sqrt{b})^2 \leq(1-\nu) a+\nu b,
\end{align}
where $r=\min\{\nu,1-\nu\}.$
\end{lemma}

\begin{theorem}\label{lala12}
Let $0<m \leq A, B \leq M$, $\Phi$ be a positive unital linear
map, $\nu\in[0,1]$ and $p>0$. Then
\begin{align}\label{eq-main}
\Phi^p\left(A\nabla_\nu B+2rMm(A^{-1}\nabla B^{-1}-A^{-1}\sharp B^{-1})\right)\leq\alpha^p\Phi^p\left(A\sharp_\nu B\right)
\end{align}
and
\begin{align}\label{eq-main1}
\Phi^p\left(A\nabla_\nu B+2rMm(A^{-1}\nabla B^{-1}-A^{-1}\sharp
B^{-1})\right)\leq\alpha^p\left(\Phi\left(A\right)\sharp_\nu
\Phi\left(B\right)\right)^p,
\end{align}
where $r=\min\{\nu,1-\nu\}$ and
$\alpha=\max\left\{\frac{(M+m)^2}{4Mm},\frac{(M+m)^2}{4^\frac{2}{p}Mm}\right\}$.
\end{theorem}
\begin{proof}

We prove first the inequalities \eqref{eq-main} and
\eqref{eq-main1} for  $0<p\leq 2$. Since $0<m \leq A, B \leq M$ we
get that
\begin{equation*}
A+MmA^{-1}\leq M+m\quad{\rm~and~}\quad B+MmB^{-1}\leq M+m.
\end{equation*}
Therefore, for a positive unital linear map $\Phi$ we have
\begin{equation*}
\Phi(A)+Mm\Phi(A^{-1})\leq M+m
\end{equation*}
and
\begin{equation*}
\Phi(B)+Mm\Phi(B^{-1})\leq M+m.
\end{equation*}
Obviously we have also the inequalities
\begin{equation*}
\Phi((1-\nu) A)+Mm\Phi((1-\nu) A^{-1})\leq (1-\nu) M+(1-\nu) m
\end{equation*}
and
\begin{equation*}
\Phi(\nu B)+Mm\Phi(\nu B^{-1})\leq \nu M+\nu m.
\end{equation*}
for any $\nu\in[0,1]$. Summing up, we therefore get
\begin{equation}\label{eq3}
\Phi(A\nabla_\nu B)+Mm\Phi((1-\nu) A^{-1}+\nu B^{-1})\leq M+m.
\end{equation}
Moreover, by using the inequality \eqref{rif-y} and functional
calculus for the positive operator
$A^{\frac{1}{2}}B^{-1}A^{\frac{1}{2}}$  we have
\begin{equation*}
 \left(A^{\frac{1}{2}}B^{-1}A^{\frac{1}{2}}\right)^\nu+r\left(A^{\frac{1}{2}}B^{-1}A^{\frac{1}{2}}+1-2
 \left(A^{\frac{1}{2}}B^{-1}A^{\frac{1}{2}}\right)^\frac{1}{2}\right)
 \leq (1-\nu)+\nu A^{\frac{1}{2}}B^{-1}A^{\frac{1}{2}}.
\end{equation*}
Multiplying both sides of the above inequality both to the left
and to the right by $A^{\frac{-1}{2}}$ we get that
\begin{equation}\label{refin1}
 A^{-1}\sharp_\nu B^{-1}+2r\left(A^{-1}\nabla B^{-1}-A^{-1}\sharp B^{-1}\right)\leq (1-\nu) A^{-1}+\nu B^{-1}.
\end{equation}
Applying \eqref{coi}, \eqref{eq3}, \eqref{refin1} and taking into
account the properties of $\Phi$ we have
\begin{align*}
&\left\|\Phi\left(A\nabla_\nu B+2rMm(A^{-1}\nabla B^{-1}-A^{-1}
\sharp B^{-1})\right)Mm\Phi^{-1}(A\sharp_\nu B)\right\|
\\&\leq\frac{1}{4}\left\|\Phi(A\nabla_\nu B+2rMm(A^{-1}\nabla B^{-1}-A^{-1}\sharp B^{-1}))
+Mm\Phi^{-1}(A\sharp_\nu
B)\right\|^2\\&\qquad\qquad\qquad\qquad\qquad (\textrm{by
Lemma\,\,\ref{main}})\\&\leq\frac{1}{4}\left\|\Phi(A\nabla_\nu
B+2rMm(A^{-1}\nabla B^{-1}-A^{-1}\sharp
B^{-1}))+Mm\Phi\big(A^{-1}\sharp_\nu
B^{-1}\big)\right\|^2\\&\qquad\qquad\qquad\qquad\qquad\qquad\qquad
(\textrm{by
inequality\,\,\eqref{coi}})\\&=\frac{1}{4}\left\|\Phi(A\nabla_\nu
B)+Mm\Phi(A^{-1}\sharp_\nu B^{-1}+2r(A^{-1}\nabla
B^{-1}-A^{-1}\sharp
B^{-1}))\right\|^2\\&\leq\frac{1}{4}\left\|\Phi(A\nabla_\nu
B)+Mm\Phi( (1-\nu) A^{-1}+\nu B^{-1})\right\|^2\quad(\textrm{by
inequality\,\,\eqref{refin1}})\\&\leq\frac{1}{4}
(M+m)^2\qquad\qquad(\textrm{by inequality\,\,\eqref{eq3}}).
\end{align*}
Therefore
 \begin{align}\label{main-man}
\left\|\Phi\left(A\nabla_\nu B+2rMm(A^{-1}\nabla
B^{-1}-A^{-1}\sharp B^{-1})\right)\Phi^{-1}(A\sharp_\nu
B)\right\|\leq \frac{(M+m)^2}{4Mm}.
\end{align}
Hence
\begin{eqnarray*}
\Phi^2\left(A\nabla_\nu B+2rMm(A^{-1}\nabla B^{-1}-A^{-1}
\sharp B^{-1})\right)\leq\left(\frac{(M+m)^2}{4Mm}\right)^2\Phi^2\left(A\sharp_\nu B\right).
\end{eqnarray*}
Since $0<p/2\le 1$, by inequality \eqref{12345a} we have
 \begin{eqnarray*}
\Phi^p\left(A\nabla_\nu B+2rMm(A^{-1}\nabla B^{-1}-A^{-1}
\sharp B^{-1})\right)\leq\left(\frac{(M+m)^2}{4Mm}\right)^p\Phi^p\left(A\sharp_\nu B\right).
\end{eqnarray*}
Thus we get the inequality \eqref{eq-main} for $0<p\leq2$. We prove
now \eqref{eq-main1} for $0<p\leq2$. Applying Lemma \ref{main} and
then inequality \eqref{eq-main} we have
\begin{align*}
&\left\|\Phi\left(A\nabla_\nu B+2rMm(A^{-1}\nabla B^{-1}-A^{-1}\sharp B^{-1})
\right)Mm(\Phi(A)\sharp_\nu \Phi(B))^{-1}\right\|\\&\leq\frac{1}{4}\left\|
\Phi(A\nabla_\nu B+2rMm(A^{-1}\nabla B^{-1}-A^{-1}\sharp B^{-1}))+Mm(\Phi(A)\sharp_\nu \Phi(B))^{-1}\right\|^2
\\&\qquad\qquad (\textrm{by Lemma\,\,\ref{main}})\\&\leq\frac{1}{4}\left\|
\Phi(A\nabla_\nu B+2rMm(A^{-1}\nabla B^{-1}-A^{-1}\sharp B^{-1}))+Mm\Phi^{-1}(A\sharp_\nu B)
\right\|^2\\&\leq\frac{1}{4} (M+m)^2\qquad(\textrm{by inequality\,\,\eqref{main-man}}).
\end{align*}
Hence the inequality \eqref{eq-main1} for $0<p\leq2$. \\
Now, we prove the inequalities \eqref{eq-main}
and \eqref{eq-main1} for $p>2$. Then, by Lemma \ref{main} and
\ref{main2} we get
\begin{align*}
M^\frac{p}{2}&m^\frac{p}{2}\left\|\Phi^\frac{p}{2}\left(A\nabla_\nu
B+2rMm(A^{-1}\nabla B^{-1}-A^{-1}\sharp
B^{-1})\right)\Phi^\frac{-p}{2} (A\sharp_\nu
B)\right\|\\&=\left\|\Phi^\frac{p}{2}\left(A\nabla_\nu
B+2rMm(A^{-1}\nabla B^{-1}-A^{-1}\sharp B^{-1})\right)
M^\frac{p}{2}m^\frac{p}{2}\Phi^\frac{-p}{2}(A\sharp_\nu
B)\right\|\\&\leq\frac{1}{4}\left\|\Phi^\frac{p}{2} (A\nabla_\nu
B+2rMm(A^{-1}\nabla B^{-1}-A^{-1}\sharp
B^{-1}))+M^\frac{p}{2}m^\frac{p}{2}\Phi^\frac{-p}{2} (A\sharp_\nu
B)\right\|^2
\\&\leq\frac{1}{4}\left\|\Big(\Phi(A\nabla_\nu
B+2rMm(A^{-1}\nabla B^{-1}-A^{-1}\sharp
B^{-1}))+Mm\Phi^{-1}(A\sharp_\nu
B)\Big)^\frac{p}{2}\right\|^2\\&\qquad
\\&=\frac{1}{4}\left\|\Phi(A\nabla_\nu
B+2rMm(A^{-1}\nabla B^{-1}-A^{-1}\sharp
B^{-1}))+Mm\Phi^{-1}(A\sharp_\nu B)\right\|^p\\&\qquad
\\&\leq\frac{1}{4}
(M+m)^p.
\end{align*}

Hence we get the inequality \eqref{eq-main} for $p>2$. Further, we
have
\begin{align*}
M^\frac{p}{2}&m^\frac{p}{2}\left\|\Phi^\frac{p}{2}\left(A\nabla_\nu
 B+2rMm(A^{-1}\nabla B^{-1}-A^{-1}\sharp B^{-1})\right)\left(\Phi\left(A\right)
 \sharp_\nu \Phi\left(B\right)\right)^\frac{-p}{2}\right\|\\&=\left\|\Phi^\frac{p}{2}
 \left(A\nabla_\nu B+2rMm(A^{-1}\nabla B^{-1}-A^{-1}\sharp B^{-1})\right)
 M^\frac{p}{2}m^\frac{p}{2}\left(\Phi\left(A\right)\sharp_\nu \Phi\left(B\right)\right)
 ^\frac{-p}{2}\right\|\\&\leq\frac{1}{4}\left\|\Phi^\frac{p}{2}
 (A\nabla_\nu B+2rMm(A^{-1}\nabla B^{-1}-A^{-1}\sharp B^{-1}))+
 M^\frac{p}{2}m^\frac{p}{2}\left(\Phi\left(A\right)\sharp_\nu \Phi\left(B\right)\right)
 ^\frac{-p}{2}\right\|^2\\&\leq\frac{1}{4}
 \left\|\Big(\Phi(A\nabla_\nu B+2rMm(A^{-1}\nabla B^{-1}-A^{-1}\sharp B^{-1}))
 +Mm\left(\Phi\left(A\right)\sharp_\nu \Phi\left(B\right)\right)^{-1}\Big)
 ^\frac{p}{2}\right\|^2
 \\&=\frac{1}{4}\left\|\Phi(A\nabla_\nu B+2rMm(A^{-1}\nabla B^{-1}-A^{-1}\sharp B^{-1}))
 +Mm\left(\Phi\left(A\right)\sharp_\nu \Phi\left(B\right)\right)^{-1}\right\|^p
 \\&\leq\frac{1}{4}\left\|\Phi(A\nabla_\nu B+2rMm(A^{-1}\nabla B^{-1}-A^{-1}\sharp B^{-1}))
 +Mm\Phi^{-1}(A\sharp_\nu B)\right\|^p\\&\leq\frac{1}{4} (M+m)^p.
\end{align*}
Thus we get the  inequality \eqref{eq-main1} for $p>2$ and this
completes the proof of the theorem.
\end{proof}

\begin{remark}
Let $0<m \leq A, B \leq M$, $\Phi$ be a positive unital linear
map. If $0<p\leq1$, then, obviously,
\begin{eqnarray}\label{fafa}
\Phi^p\left(A\nabla_\nu B\right)\leq
\left(\Phi\left(A\nabla_\nu B\right)+2rMm\Phi\left(A^{-1}\nabla
 B^{-1}-A^{-1}\sharp B^{-1}\right)\right)^p.
\end{eqnarray}
Hence the inequality \eqref{fafa} shows that Theorem \ref{lala12}
is a refinement of inequalities \eqref{lala1} and \eqref{lala2}
for $0<p\leq1$. \\
We also have
\begin{eqnarray*}
\Phi^p\left(A\nabla_\nu B\right)\leq\Phi^p\left(A\nabla_\nu B\right)+(2rMm)^p\Phi^p\left(A^{-1}\nabla B^{-1}-A^{-1}\sharp B^{-1}\right),
\end{eqnarray*}
where $p\ge 1$,  $\nu\in[0,1]$ and $r=\min\{\nu,1-\nu\}$.

Hence
\begin{align*}
\left\|\Phi^p\left(A\nabla_\nu
B\right)\right\|&\leq\left\|\Phi^p\left(A\nabla_\nu B\right)
+(2rMm)^p\Phi^p\left(A^{-1}\nabla B^{-1}-A^{-1}\sharp
B^{-1}\right)\right\|\\&\leq\left \|\Phi^p\Big(A\nabla_\nu
B+2rMmA^{-1}\nabla B^{-1}-A^{-1}\sharp
B^{-1}\Big)\right\|\qquad(\textrm{by Lemma\,\,\ref{main2}}).
\end{align*}

Therefore, Theorem \ref{lala12} is a refinement of the
inequalities, \eqref{123456789a} and
\eqref{123456789ab} for the operator norm and $p\ge 2$.
\end{remark}

The following examples show that inequality \eqref{eq-main} is a
refinement of \eqref{lala1} and \eqref{123456789a}.
\begin{example}
If   $A=\left(\begin{array}{cc}
         1.75 &  0.433\\
0.433  & 1.25
 \end{array}\right)$, $B=\left(\begin{array}{cc}
     2.5 &  0.5\\
0.5 & 2.5
 \end{array}\right)$, $\Phi(X)=\frac{1}{2}\rm{tr}(X)\,\,(X\in\mathbb{M}_2)$, $m=1$, $M=3$, $\nu=\frac{1}{2}$ and $p=3$, then
$ A\nabla_\nu B=\left(\begin{array}{cc}
  2.1250  &  0.4665\\
    0.4665 &   1.8750
 \end{array}\right)$ and
  $A\nabla_\nu B+2rMm\left(A^{-1}\nabla B^{-1}-A^{-1}\sharp B^{-1}\right)=\left(\begin{array}{cc}
  2.1601  &  0.4260\\
    0.4260  &  2.0016
 \end{array}\right)$. Hence
 {\footnotesize\begin{align*}
 \Phi^3\left(A\nabla_\nu B+2rMm\left(A^{-1}\nabla B^{-1}-A^{-1}\sharp B^{-1}\right)\right)-\Phi^3(A\nabla_\nu B)=9.0095-8=1.0095>0.
  \end{align*}}
\end{example}
\begin{example}
Let  $\Phi(X)=T^*XT\,\,(X\in\mathbb{M}_2)$, where $T=\left(\begin{array}{cc}
         \frac{\sqrt{2}}{2} &  \frac{\sqrt{2}}{2}\\
-\frac{\sqrt{2}}{2}  & \frac{\sqrt{2}}{2}
 \end{array}\right)$. If  $A=\left(\begin{array}{cc}
         5 &  -2\\
-2  & 5
 \end{array}\right)$, $B=\left(\begin{array}{cc}
     4.75&    0.433\\
0.433&    4.25
 \end{array}\right)$, $m=3$, $M=7$, $\nu=\frac{1}{2}$ and $p=\frac{5}{3}$, then
$ A\nabla_\nu B=\left(\begin{array}{cc}
   4.8750  & -0.7835\\
   -0.7835  &  4.6250
 \end{array}\right)$ and
  $A\nabla_\nu B+2rMm\left(A^{-1}\nabla B^{-1}-A^{-1}\sharp B^{-1}\right)=\left(\begin{array}{cc}
    5.0283  & -0.7730\\
   -0.7730   & 4.7909
 \end{array}\right)$. Hence
 {\footnotesize\begin{align*}
 \Phi^\frac{5}{3}\left(A\nabla_\nu B+2rMm\left(A^{-1}\nabla B^{-1}-A^{-1}\sharp B^{-1}\right)\right)-\Phi^\frac{5}{3}(A\nabla_\nu B)=\left(\begin{array}{cc}
   0.7838 & -1.0172\\
   -1.0172   & 0.7199
 \end{array}\right)>0.
  \end{align*}}
 \end{example}

\begin{corollary}
Let $0<m \leq A, B \leq M$ and $\Phi$ be a positive unital linear map. Then
\begin{eqnarray*}
\Phi^p\left(\frac{A+B}{2}+Mm(A^{-1}\nabla B^{-1}-A^{-1}\sharp
B^{-1})\right)\leq{\alpha}^p\Phi^p\left(A\sharp B\right)
\end{eqnarray*}
and
\begin{eqnarray*}
\Phi^p\left(\frac{A+B}{2}+Mm(A^{-1}\nabla B^{-1}-A^{-1}\sharp
B^{-1})\right)\leq{\alpha}^p\left(\Phi\left(A\right)\sharp
\Phi\left(B\right)\right)^p.
\end{eqnarray*}
\end{corollary}
\begin{proof}
Take $r=\nu=\frac{1}{2}$ in Theorem \ref{lala12}.
\end{proof}

If the positive unital linear map $\Phi(A)=A\,\,(A\in{\mathbb
B}({\mathscr H}))$, then we get from Theorem \ref{lala12} the
following reverse AM--GM inequalities, which improve the reversed
AM--GM inequality \eqref{mos1}.

\begin{corollary}
Let $0<m \leq A, B \leq M$. Then, the inequalities
\begin{eqnarray*}
\left(\frac{A+B}{2}+Mm(A^{-1}\nabla B^{-1}-A^{-1}\sharp
B^{-1})\right)^p\leq
\left(\frac{(M+m)^2}{4Mm}\right)^p\left(A\sharp
B\right)^p\,\,\,\,(0< p\leq 2)
\end{eqnarray*}
and
\begin{eqnarray*}
\left(\frac{A+B}{2}+Mm(A^{-1}\nabla B^{-1}-A^{-1}\sharp
B^{-1})\right)^p\leq
\left(\frac{(M+m)^2}{4^{2/p}Mm}\right)^p\left(A\sharp
B\right)^p\,\,\,\,(p>2).
\end{eqnarray*}
hold.
\end{corollary}
The operator P\'olya--Szeg\"o inequality states that
\begin{align}\label{polya}
\Phi(A)\sharp\Phi(B)\leq \frac{M+m}{2\sqrt{mM}}\Phi(A\sharp B).
\end{align}
where $0<m_1^2\leq A\leq M_1^2$, $0<m_2^2\leq B\leq M_2^2$, $m={m_2\over M_1}$ and  $M={M_1\over m_2}$. Also the operator Kantorovich inequality says that
\begin{align}\label{kantoro}
\Phi(A)\sharp\Phi(A^{-1})\leq \frac{M^2+m^2}{2{mM}},
\end{align}
where $0<m_1^2\leq A\leq M_1^2$, $0<m_2^2\leq B\leq M_2^2$, $m={m_2\over M_1}$, $M={M_1\over m_2}$; see \cite{kmm}. \\
In the following result we show some refinements of  \eqref{polya} and \eqref{kantoro}.
\begin{theorem}
Let $\Phi$ be a unital positive linear map, $0<m_1^2\leq A\leq M_1^2$, $0<m_2^2\leq B\leq M_2^2$, $m={m_2\over M_1}$, $M={M_1\over m_2}$.
{\footnotesize\begin{align}\label{haha}
\Phi(A)\sharp\Phi(B)+\frac{1}{2}\left(\sqrt{Mm}\Phi(A)+\frac{1}{{\sqrt{Mm}}}\Phi(B)-2\left(\Phi(A)\sharp\Phi(B)\right)\right)\leq \frac{M+m}{2\sqrt{mM}}\Phi(A\sharp B).
\end{align}}
In particular, if $B=A^{-1}$, then
{\footnotesize\begin{align*}
\Phi(A)\sharp\Phi(A^{-1})+\frac{1}{2}\left({Mm}\Phi(A)+\frac{1}{{Mm}}\Phi(A^{-1})-2\left(\Phi(A)\sharp\Phi(A^{-1})\right)\right)\leq \frac{M^2+m^2}{2{mM}}.
\end{align*}}
\end{theorem}
\begin{proof}
If $0<m_1^2\leq A\leq M_1^2$ and  $0<m_2^2\leq B\leq M_2^2$, then

\begin{align*}
m^2={m_2^2\over M_1^2}\leq A^{-1\over2}BA^{-1\over2}\leq{M_1^2\over m_2^2}=M^2,
\end{align*}
whence
\begin{align*}
\left(M-\left(A^{-1\over2}BA^{-1\over2}\right)^{1\over2}\right)\left(\left(A^{-1\over2}BA^{-1\over2}\right)^{1\over2}-m\right)\geq0.
\end{align*}
Hence
\begin{align*}
MmA+B\leq (M+m)A\sharp B,
\end{align*}
whence
\begin{align}\label{saqse}
Mm\Phi(A)+\Phi(B)\leq(M+m) \Phi(A\sharp B).
\end{align}
Using lemma \ref{rif-y} for the operators $Mm\Phi(A)$, $\Phi(B)$ and $\nu={1\over2}$ we get
{\footnotesize\begin{align}\label{saqse2}
\sqrt{Mm}\left(\Phi(A)\sharp\Phi(B)\right)+{1\over2}\left(Mm\Phi(A)+\Phi(B)-2\sqrt{Mm}\left(\Phi(A)\sharp\Phi(B)\right)\right)\leq {1\over2}\left(Mm\Phi(A)+\Phi(B)\right).
\end{align}}
Applying inequalities \eqref{saqse} and \eqref{saqse2} we get the first inequality.
 In particular, if we consider $m_1^2=m^2\leq A\leq M^2=M_1^2$, then by putting $m_2^2={1\over M^2}\leq A^{-1}\leq{1\over m^2}=M_2^2$ in \eqref{haha} we reach the desired inequality.
\end{proof}
If we take $\Phi$ in \eqref{haha} to be the positive linear map defined on the diagonal blocks of operators by $\Phi({\rm diag}(A_1, \cdots, A_n))=\frac{1}{n}\sum_{j=1}^nA_j$, then we get the following refinements of a reversed Cauchy-Schwarz operator inequality.
\begin{corollary}
Let $0<m_1^2\leq A_j\leq M_1^2$, $0<m_2^2\leq B_j\leq M_2^2$$\,\,(1\leq j\leq n)$, $m={m_2\over M_1}$, $M={M_1\over m_2}$. Then
{\begin{align*}
\left(\sum_{ j=1}^n A_j\sharp\sum_{ j=1}^n B_j\right)+\frac{1}{2}\left(\sqrt{Mm}\sum_{ j=1}^n A_j\frac{1}{{\sqrt{Mm}}}\sum_{ j=1}^n B_j-\right.&\left.2\left(\sum_{ j=1}^n A_j\sharp\sum_{ j=1}^n B_j\right)\right)\nonumber\\&\leq \frac{M+m}{2\sqrt{mM}}\left(\sum_{ j=1}^n A_j\sharp B_j\right).
\end{align*}}

\end{corollary}
\begin{proposition}
Let $0<m\leq A\leq M$ and $x\in {\mathscr H}$. Then
{\footnotesize\begin{align*}\label{saqse}
\langle Ax,x\rangle^{1\over2}\langle A^{-1}x,x\rangle^{1\over2}+{1\over2}\left(\sqrt[4]{Mm}\langle Ax,x\rangle^{1\over2}-{1\over\sqrt[4]{Mm}}\langle A^{-1}x,x\rangle^{1\over2}\right)^2\leq{M+m\over2\sqrt{Mm}}\langle x,x\rangle^2.
\end{align*}}
\end{proposition}
\bigskip
%===================================================================================================================================
\bibliographystyle{amsplain}

\end{document}